\documentclass[12pt]{article}

\usepackage{graphicx,tikz,color,url}
\usepackage{amsmath,amssymb,latexsym}
\usepackage{cite}

\newtheorem{theorem}{Theorem}[section]

\newtheorem{lemma}[theorem]{Lemma}
\newtheorem{proposition}[theorem]{Proposition}

\newcommand{\proof}{\noindent{\bf Proof.\ }}
\newcommand{\qed}{\hfill $\square$ \bigskip}

\newcommand{\diam}{{\rm diam}}
\newcommand{\pch}{\chi_{\rho}}
\newcommand{\pchu}{\widehat{\chi}_{\rho}}

\begin{document}

\title{On Sierpi\'{n}ski packing chromatic number and recognition of Sierpi\'{n}ski products}

\author{P\v{r}emysl Holub$^{a,}$\thanks{Email: \texttt{holubpre@kma.zcu.cz}} 
\and Sandi Klav\v zar $^{b,c,d,}$\thanks{Email: \texttt{sandi.klavzar@fmf.uni-lj.si}}
}
\maketitle

\begin{center}
$^a$ Faculty of Applied Sciences, University of West Bohemia, Pilsen, Czech Republic\\
\medskip

$^b$ Faculty of Mathematics and Physics, University of Ljubljana, Slovenia\\
\medskip

$^c$ Institute of Mathematics, Physics and Mechanics, Ljubljana, Slovenia\\
\medskip

$^d$ Faculty of Natural Sciences and Mathematics, University of Maribor, Slovenia
\end{center}

\begin{abstract}
The Sierpi\'{n}ski product $G \otimes _f H$ of graphs $G$ and $H$ with respect to a function $f \colon V(G)\rightarrow V(H)$ has the vertex set $V(G)\times V(H)$. For every $g\in V(G)$ it contains a disjoint copy $gH$ of $H$, and for every edge $gg'$ of $G$ there is the edge $(g,f(g'))(g',f(g))$ between $gH$ and $g'H$. In this paper, the  Sierpi\'{n}ski packing chromatic number is defined as the minimum of $\chi_{\rho}(G\otimes _f H)$ over all functions $f$, where $\chi_{\rho}(X)$ is the packing chromatic number of $X$. The upper Sierpi\'{n}ski packing chromatic number is analogously defined as the maximum corresponding value. The (upper) Sierpi\'{n}ski packing chromatic number is determined for all Sierpi\'{n}ski product graphs whose both factors are complete. Sierpi\'{n}ski product graphs whose factors are paths or stars are also studied. Their Sierpi\'{n}ski packing chromatic number is always $3$, while their upper Sierpi\'{n}ski packing chromatic number is bounded from below and above. It is also proved that for a given graph $G$, it can be checked in polynomial time whether $G$ has a representation as a Sierpi\'{n}ski product graphs both factors of which being trees.
\end{abstract}

\noindent
{\bf Keywords:} Sierpi\'{n}ski product; packing chromatic number; complete graph; tree; recognition algorithm

\medskip

\noindent
{\bf AMS Subj.\ Class.\ (2020)}: 05C76, 05C15, 05C12, 68Q25

%%%%%%%%%%%%%%%%%%%%%%%%%%%%%%%%%%%%%%%%%%%%%%%%%%%%%
\section{Introduction}
\label{sec:intro}
%%%%%%%%%%%%%%%%%%%%%%%%%%%%%%%%%%%%%%%%%%%%%%%%%%%%%

Sierpi\'{n}ski graphs represent a very interesting and widely studied family of graphs. They were defined in 1997 in the paper~\cite{klavzar-1997}, where the primary motivation for their introduction was the intrinsic link to the Tower of Hanoi problem. Intensive research of Sierpi\'{n}ski graphs led to a review article~\cite{hinz-2017} in which state of the art up to 2017 is summarized and unified approach to Sierpi\'{n}ski-type graph families is also proposed. Later research on Sierpi\'{n}ski graphs includes~\cite{balakrishnan-2022, chanda-2025, farrokhi-2021, liu-2021, menon-2023, varghese-2021, zhang-2024}.

An interesting generalization of Sierpi\'{n}ski graphs in a completely new direction has been proposed in 2023 by Kovi\v{c}, Pisanski, Zemlji\v{c}, and \v{Z}itnik in~\cite{kpzz-2023}. In the spirit of classical graph products,  where the vertex set of a product graph is the Cartesian product of the vertex sets of the factors~\cite{hik-2011}, they introduced the Sierpi\'{n}ski product of graphs. In this set-up,  the classical, second power Sierpi\'{n}ski graphs appear as a special case. Moreover, it is also shown how to extend the Sierpi\'{n}ski product to multiple factors. 

In the seminal paper~\cite{kpzz-2023}, several basic properties of Sierpi\'{n}ski products are established. For instance, conditions on the factor graphs that yield planar Sierpi\'{n}ski products are presented, and a lot of effort is put into understanding the symmetries of this graph operation. Very recently, in the paper~\cite{mafnucci-2025+}, Maffucci is undertaking a new, in-depth study of the structure of Sierpi\'{n}ski product graphs. It reports necessary and sufficient conditions for their higher connectivity and describes the complete classification of the regular polyhedral Sierpi\'ski products. 

Assume now that an arbitrary graph invariant is given. Since any two fixed graphs give rise to many non-isomorphic Sierpi\'{n}ski products, it makes sense to consider the best and the worst behavior of the invariant over all such product graphs. Following this idea, the (upper) Sierpi\'{n}ski domination number and the (upper) Sierpi\'{n}ski metric dimension were respectively investigated in~\cite{henning-2023a+, henning-2023b+}. In the latter paper some metric properties of Sierpi\'{n}ski products are also determined, in particular, it is proved that the layers with respect to the second factor in a Sierpi\'{n}ski product graph are convex. Sierpi\'{n}ski products of complete graphs were further explored in\cite{tian-2025} where the focus was on their Sierpi\'nski general position number and their lower Sierpi\'nski general position number. 

The packing chromatic number was introduced in~\cite{goddard-2008}, named with the present name in~\cite{bresar-2007}, and extensively studied afterwards. A summary of the situation up to 2020 is given in the survey paper~\cite{bresar-2020}, for some recent developments in the area we refer to~\cite{dliou-2025, ferme-2025, gozupek-2025, gregor-2024, grochowski-2025} and references therein. From our perspective, the most important thing is that the packing chromatic number of Sierpi\'{n}ski graphs was investigated in the series of papers~\cite{bresar-2018, bresar-2016, deng-2021, korze-2019}. 

In this paper we continue the investigation of the Sierpi\'{n}ski product graphs with the main focus on their packing chromatic number. In the next section we recall some definitions, formally introduce the Sierpi\'{n}ski product, and define the (upper) Sierpi\'{n}ski packing chromatic number. In Section~\ref{sec:Hamming}, we determine the Sierpi\'{n}ski packing chromatic number and the upper Sierpi\'{n}ski packing chromatic number of the most dense Sierpi\'{n}ski product graphs, that is, those whose both factors are complete. In  Section~\ref{sec:paths}, we consider Sierpi\'{n}ski product graphs whose factors are paths or stars. For all of these four variants it is shown that their Sierpi\'{n}ski packing chromatic number is always $3$, while their upper Sierpi\'{n}ski packing chromatic numbers are bounded from below and above. In Section~\ref{sec:recognition}, we prove that for a given graph $G$, it can be checked in polynomial time whether $G$ has a representation as a Sierpi\'{n}ski product graphs both factors of which being trees.
% We conclude the paper with several open problems and directions for future investigation. 

%%%%%%%%%%%%%%%%%%%%%%%%%%%%%%%%%%%%%%%%%%%%%%%%%%
\section{Preliminaries}
%%%%%%%%%%%%%%%%%%%%%%%%%%%%%%%%%%%%%%%%%%%%%%%%%%

Let $G = (V(G), E(G))$ be a connected graph. If $u, v\in V(G)$, then $d_G(u,v)$ stands for the shortest-path distance between $u$ and $v$ in $G$. By a ($u,v$)-path in $G$ we mean a path between $u$ and $v$ in $G$. The diameter of $G$ is the largest distance between two vertices of $G$ and is denoted by $\diam(G)$. Let $X\subseteq V(G)$. Then $X$ is {\em independent} if $d_G(u,v) \ge 2$ holds for each pair of different vertices $u$ and $v$ from $X$. The {\em independence number} $\alpha(G)$ of $G$ is the cardinality of a largest independent set. If $d_G(u,v) \ge 3$ holds for each pair of different vertices $u$ and $v$ from $X$, then  $X$ is a {\em $2$-packing} of $G$, and the cardinality of a largest $2$-packing of $G$ is denoted by $\alpha_2(G)$.

Let $G$ and $H$ be graphs and let $f \colon V(G)\rightarrow V(H)$ be an arbitrary function. We will briefly write $f\in H^G$ for $f \colon V(G)\rightarrow V(H)$. The \textit{Sierpi\'{n}ski product of graphs $G$ and $H$ with respect to $f$}, denoted by  $G \otimes _f H$, is defined as the graph on the vertex set $V(G)\times V(H)$ with edges of two types:
\begin{itemize}
    \item \emph{Type-$1$ edges}:
    $(g,h)(g,h')$ is an edge of  $G \otimes _f H$ for every vertex $g\in V(G)$ and every edge $hh' \in E(H)$,
    \item \emph{Type-$2$ edges}: $(g,f(g'))(g',f(g))$ is an edge of $G \otimes _f H$ for every edge $gg' \in E(G)$.
\end{itemize}
We say that the graph $G$ is the {\em base graph} of $G \otimes _f H$ and that $H$ is the {\em fiber graph} of $G \otimes _f H$.  

We observe that the edges of Type-$1$ induce $n(G)$ copies of the graph $H$ in the Sierpi\'{n}ski product $G \otimes_f H$. For each vertex $u \in V(G)$, we denote by $uH$ the copy of $H$ corresponding to the vertex $u$. A Type-$2$ edge joins vertices from different copies of $H$ in $G \otimes _f H$, and is called a \emph{connecting edge} of $G \otimes _f H$. A vertex incident with a connecting edge is called a \emph{connecting vertex}. Note that for $u\in V(G)$, the fiber $uH$ of $G \otimes_f H$ contains at most $\deg_G(u)$ connecting vertices. 

The {\em packing chromatic number} $\pch(G)$ of a graph $G$ is the smallest integer $k$ for which there exists a mapping $c:V(G)\rightarrow [k] = \{1,\ldots, k\}$, such that if $c(u) = c(v) = \ell$ for $u\ne v$, then $d_G(u,v) > \ell$. Such a mapping $c$ is called a {\em packing $k$-coloring}. A monochromatic class of vertices colored with color $i$ is called an {\em $i$-packing}; note that each pair of vertices of an $i$-packing is mutually at distance more than $i$. 

Let $G$ and $H$ be graphs and let 
$$H^G = \{f:\ f:V(G) \rightarrow V(H)\}\,.$$ 
The \textit{Sierpi\'{n}ski packing chromatic number} $\pch(G, H)$ of the pair $(G,H)$ is
\[
\pch(G, H) = \min_{f\in H^G}\{\pch(G\otimes _f H)\}\,,
\]
and the \textit{upper Sierpi\'{n}ski packing chromatic number} $\pchu(G, H)$ of the pair $(G,H)$ is
\[
\pchu(G, H) = \max_{f\in H^G}\{\pch(G\otimes _f H)\}\,.
\]

If $G$ is a graph and $k\ge 1$, then the $k$-corona $G \odot pK_1$ of $G$ is the graph obtained from $G$ by adding $k$ pendant vertices to every vertex of $G$. We will need the following results below due to La\"{\i}che, Bouchemakh, and Sopena. 

\iffalse
\begin{theorem} {\rm \cite[Theorem 4]{laiche-2017}}
\label{thm:laiche K_1}
If $n\ge 1$, then 
\[
\pch(P_n\odot K_1) = \left.
  \begin{cases}
    2; & n = 1, \\
    3; & n \in \{2,3\}, \\
    4; & 4\le n \le 9, \\
    5; & n \ge 10.
  \end{cases}
  \right.
\]
\end{theorem}
\fi

\begin{theorem} {\rm \cite[Theorem 8]{laiche-2017}}
\label{thm:laiche}
If $n\ge 1$, then 
\[
\pch(P_n\odot 2K_1) = \left.
  \begin{cases}
    2; & n = 1, \\
    3; & n = 2, \\
    4; & n\in \{3,4\}, \\
    5; & 5\le n\le 11, \\
    6; & n\ge 12.
  \end{cases}
  \right.
\]
\end{theorem}

\begin{theorem} {\rm \cite[Theorem 9]{laiche-2017}}
\label{thm:laiche 3K_1}
If $n\ge 1$, then 
\[
\pch(P_n\odot 3K_1) = \left.
  \begin{cases}
    2; & n = 1, \\
    3; & n =2, \\
    4; & \in \{3,4\} , \\
    5; & 5\le n \le 8, \\
    6; & n \ge 9.
  \end{cases}
  \right.
\]
\end{theorem}

\begin{theorem} {\rm \cite[Theorem 7]{laiche-2017}}
\label{thm:laiche pK_1}
If $n\ge 1$ and $p\geq 4$, then 
\[
\pch(P_n\odot pK_1) = \left.
  \begin{cases}
    2; & n = 1, \\
    3; & n =2, \\
    4; & \in \{3,4\}, \\
    5; & 5\le n \le 8, \\
    6; & 9 \le n \le 34, \\
    7; & n \ge 35.
  \end{cases}
  \right.
\]
\end{theorem}

%%%%%%%%%%%%%%%%%%%%%%%%%%%%%%%%%%%%%%%%%%%%%%%%%sec:Hamming%
\section{Sierpi\'{n}ski products of complete graphs}
\label{sec:Hamming}
%%%%%%%%%%%%%%%%%%%%%%%%%%%%%%%%%%%%%%%%%%%%%%%%%%

In this section we determine $\pch(K_m, K_n)$ and $\pchu(K_m, K_n)$ for all $m\ge 2$ and $n\ge 2$. Throughout this section we assume that $V(K_m)=\{u_1,\dots, u_m\}$ and that $V(K_n) = [n]$. For an example of the Sierpi\'{n}ski product of complete graphs together with the notation used in this section, see Fig.~\ref{fig:K_5-K_4-example}.

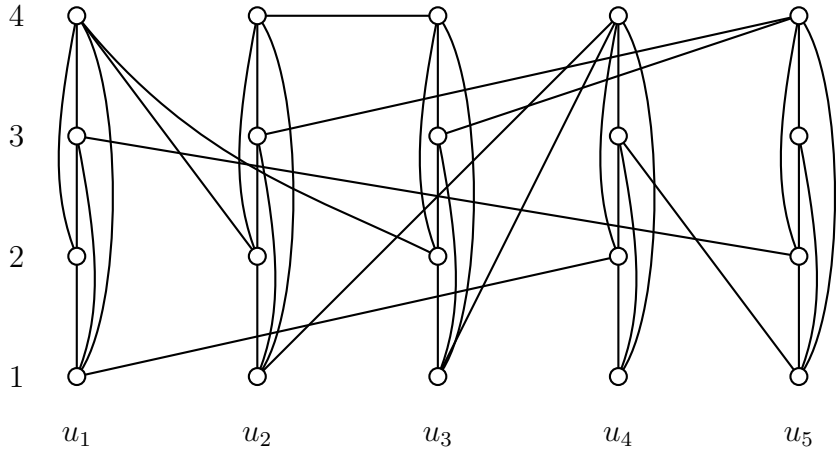
\begin{figure}[ht!]
\begin{center}
\begin{tikzpicture}[scale=0.8,style=thick,x=1cm,y=1cm]
\def\vr{4pt}

\begin{scope}[xshift=0cm, yshift=0cm] % K4 1
\coordinate(x1) at (0,0);
\coordinate(x2) at (0,2);
\coordinate(x3) at (0,4);
\coordinate(x4) at (0,6);
% edges		
\draw (x1) -- (x2) -- (x3) -- (x4);
\draw (x1) .. controls (0.4,1) and (0.4,2) .. (x3);
\draw (x2) .. controls (-0.4,3) and (-0.4,4) .. (x4);
\draw (x1) .. controls (0.8,1) and (0.8,5) .. (x4);
% connecting edges
% from u_1
\draw (x4) -- (3,2);
\draw (x4) .. controls (2,3.5) and (4,3) .. (6,2);
\draw (x1) -- (9,2);
\draw (x3) -- (12,2);
% from u_2
\draw (3,6) -- (6,6);
\draw (3,0) -- (9,6);
\draw (3,4) -- (12,6);
% from u_3
\draw (6,0) -- (9,6);
\draw (6,4) -- (12,6);
% from u_4
\draw (9,4) -- (12,0);
%  vertices
\draw(x1)[fill=white] circle(\vr);
\draw(x2)[fill=white] circle(\vr);
\draw(x3)[fill=white] circle(\vr);
\draw(x4)[fill=white] circle(\vr);
% text
\node at (0,-1) {$u_1$};
\node at (3,-1) {$u_2$};
\node at (6,-1) {$u_3$};
\node at (9,-1) {$u_4$};
\node at (12,-1) {$u_5$};
\node at (-1,0) {$1$};
\node at (-1,2) {$2$};
\node at (-1,4) {$3$};
\node at (-1,6) {$4$};
\end{scope}
\begin{scope}[xshift=3cm, yshift=0cm] % K4 2
\coordinate(x1) at (0,0);
\coordinate(x2) at (0,2);
\coordinate(x3) at (0,4);
\coordinate(x4) at (0,6);
% \edges		
\draw (x1) -- (x2) -- (x3) -- (x4);
\draw (x1) .. controls (0.4,1) and (0.4,2) .. (x3);
\draw (x2) .. controls (-0.4,3) and (-0.4,4) .. (x4);
\draw (x1) .. controls (0.8,1) and (0.8,5) .. (x4);
%  vertices
\draw(x1)[fill=white] circle(\vr);
\draw(x2)[fill=white] circle(\vr);
\draw(x3)[fill=white] circle(\vr);
\draw(x4)[fill=white] circle(\vr);
\end{scope}
\begin{scope}[xshift=6cm, yshift=0cm] % K4 3
\coordinate(x1) at (0,0);
\coordinate(x2) at (0,2);
\coordinate(x3) at (0,4);
\coordinate(x4) at (0,6);
% \edges		
\draw (x1) -- (x2) -- (x3) -- (x4);
\draw (x1) .. controls (0.4,1) and (0.4,2) .. (x3);
\draw (x2) .. controls (-0.4,3) and (-0.4,4) .. (x4);
\draw (x1) .. controls (0.8,1) and (0.8,5) .. (x4);
%  vertices
\draw(x1)[fill=white] circle(\vr);
\draw(x2)[fill=white] circle(\vr);
\draw(x3)[fill=white] circle(\vr);
\draw(x4)[fill=white] circle(\vr);
\end{scope}
\begin{scope}[xshift=9cm, yshift=0cm] % K4 4
\coordinate(x1) at (0,0);
\coordinate(x2) at (0,2);
\coordinate(x3) at (0,4);
\coordinate(x4) at (0,6);
% \edges		
\draw (x1) -- (x2) -- (x3) -- (x4);
\draw (x1) .. controls (0.4,1) and (0.4,2) .. (x3);
\draw (x2) .. controls (-0.4,3) and (-0.4,4) .. (x4);
\draw (x1) .. controls (0.8,1) and (0.8,5) .. (x4);
%  vertices
\draw(x1)[fill=white] circle(\vr);
\draw(x2)[fill=white] circle(\vr);
\draw(x3)[fill=white] circle(\vr);
\draw(x4)[fill=white] circle(\vr);
\end{scope}
\begin{scope}[xshift=12cm, yshift=0cm] % K4 5
\coordinate(x1) at (0,0);
\coordinate(x2) at (0,2);
\coordinate(x3) at (0,4);
\coordinate(x4) at (0,6);
% \edges		
\draw (x1) -- (x2) -- (x3) -- (x4);
\draw (x1) .. controls (0.4,1) and (0.4,2) .. (x3);
\draw (x2) .. controls (-0.4,3) and (-0.4,4) .. (x4);
\draw (x1) .. controls (0.8,1) and (0.8,5) .. (x4);
%  vertices
\draw(x1)[fill=white] circle(\vr);
\draw(x2)[fill=white] circle(\vr);
\draw(x3)[fill=white] circle(\vr);
\draw(x4)[fill=white] circle(\vr);
\end{scope}

\end{tikzpicture}
\caption{The Sierpi\'{n}ski product $K_5 \otimes _f K_4$, where $f(u_1) = 2$, $f(u_2) = 4$, $f(u_3) = 4$, $f(u_4) = 1$, $f(u_5) = 3$.}
	\label{fig:K_5-K_4-example}
\end{center}
\end{figure}

We start with a couple of lemmas. 

\begin{lemma} \label{lem:diam}
If $m,n\geq 2$ and $f\in K_n^{K_m}$, then $\diam(K_m \otimes_f K_n)=3$.
\end{lemma}

\begin{proof}
Let $G=K_m \otimes_f K_n$. By the definition of the Sierpi\' nski product, there is a connecting edge between any pair of fibers since $K_m$ is complete. In addition, since $K_n$ is complete, every vertex of any fiber is at distance at most  $1$ from any connecting edge. Therefore $\diam(G)\leq 3$. Assume first that $f$ is a constant function with value e.g.\ $i\in V(K_n)$). Then for every different fibers $u_jK_n$, $u_{\ell}K_n$, there is a pair of vertices $(u_j,i')\in u_jK_n$, $(u_{\ell},i'')\in u_{\ell}K_n$, $i'\not= i''$, which are at distance $3$. Assume second that $f$ is not a constant function and let $f(u_i)=j$ and $f(u_{i'}) = j'$, where $j\not= j'$. Then $(u_j,j')$ and $(u_i,j)$ are not incident to the connecting edge between $u_jK_n$ and $u_iK_n$. Suppose that $(u_j,j')$ and $(u_i,j)$ have a common neighbor. Then this common neighbor belongs to some fiber $u_kK_n$, $k\not=u_j$, $w\not=u_i$. Let $(u_k,\ell)$ be such a vertex. By the definition of the Sierpi\' nski product, $f(u_j)=f(u_i)=\ell$, a contradiction. Therefore $d_G((u_j,j'), (u_i,j))=3$. 
\qed
\end{proof}

\begin{lemma}
\label{lem:alpha2}
If $n\ge 3$, $n\ge m\ge 2$, and $f\in K_n^{K_m}$, then 
$\alpha_2(K_m \otimes_f K_n) = m$. 
\end{lemma}

\begin{proof}
Since $n \ge m$, at most $n-1$ vertices of any fiber $u_iK_n$ are connecting vertices. Thus, in each fiber, at least one vertex is not connecting; let $M$ denote a set of non-connecting vertices such that $M$ contains exactly one vertex from each fiber $u_iK_n$, $i\in [m]$). Hence $\vert M\vert =m$. Since the vertices of $M$ are pairwise at distance at least $3$, we have $\alpha_2(K_m \otimes_f K_n)\geq m$. On the other hand, no $2$-packing can contain two vertices of the same fiber. Therefore $\alpha_2(K_m \otimes_f K_n)= m$. 
\qed
\end{proof}

\begin{lemma}
\label{lem:pch-of-hamming}
If $m\ge 3$, $n\ge 3$, and $f\in K_n^{K_m}$, then  
$$\pch(K_m \otimes_f K_n) \ge mn - 2m   + 2\,.$$ 
Moreover, if $n\ge m$, then the equality holds. 
\end{lemma}

\begin{proof}
Let $G = K_m \otimes_f K_n$. Since every fiber $u_iK_n$ is complete, we have $\alpha(G)\le m$ and $\alpha_2(G)\le m$. This means that in any packing coloring of $G$, at most $m$ vertices can receive color $1$ and at most $m$ vertices can receive color $2$. By Lemma~\ref{lem:diam}, every other vertex requires a new, private color. It follows that $\pch(G) \ge 2 + (mn - 2m) = mn - 2m + 2$.

Assume now that $n\ge m$ and distinguish the following two cases. 

\medskip\noindent
{\bf Case 1}: $f$ is not surjective. \\
We may  without loss of generality assume that the vertex $n$ is an image of no vertex from $K_m$. Then the set $M = \{(u_1,n), (u_2,n), \ldots, (u_m,n)\}$ is a $2$-packing of $G$. Consider the Sierpi\'nski product $G' = K_m \otimes_{f'} K_{n-1}$ which is the subgraph of $G$ induced by the vertices from $V(G)\setminus M$, and where $f' = f$. Then $G'$ contains a $2$-packing $N$ of size $m$ by the virtue of Lemma~\ref{lem:alpha2}. As $G'$ is induced in $G$, the set $N$ is also an independent set of $G$, hence we may color all vertices of $N$ with color $1$. Coloring all the remaining (uncolored) vertices of $G$ with mutually distinct colors, we have a packing coloring of $G$ with $mn - 2m   + 2$ colors, so that $\pch(G) \le mn - 2m + 2$. We can conclude that $\pch(G) = mn - 2m + 2$. 

\medskip\noindent
{\bf Case 2}: $f$ is surjective. \\
In this case $n=m$, and we may without loss of generality assume that $f(u_i) = i$, $i\in [m]$. Set $M = \{(u_1,1), (u_2,2), \ldots, (u_m,m)\}$ and observe that $M$ is  a 2-packing, hence coloring all vertices of $G$ with color $2$ is legal. Moreover, we infer that the set $N = \{(u_1,2), (u_2,3), \ldots, (u_m,1)\}$ is an independent set of $G$, thus we may color all vertices of $N$ with color $1$. Now just as in Case 1, by coloring all the remaining vertices of $G$ with mutually distinct colors, we find that $\pch(G) \le mn - 2m + 2$, and therefore $\pch(G) = mn - 2m + 2$. 
\qed
\end{proof}

\begin{theorem}
\label{thm:pch-hamming}
If $m\ge 3$ and $n\ge 3$, then  
$$\pch(K_m, K_n) = mn - 2m   + 2\,.$$ 
\end{theorem}

\begin{proof}
Lemma~\ref{lem:pch-of-hamming} implies that $\pch(K_m, K_n) \ge mn - 2m + 2$. We thus need to prove that for every $m\ge 3$ and every $n\ge 3$ there exists $f\in K_n^{K_m}$ such that $\pch(K_m \otimes_f K_n) = mn - 2m   + 2$. If $n\ge m$, then by Lemma~\ref{lem:pch-of-hamming}, any $f\in K_n^{K_m}$ does the job. 

Assume now that $m\ge n\ge 3$. Consider the constant function $f'\in K_n^{K_m}$ defined by $f'(u_i) = 1$ for each $i\in [m]$. Then $X= \{(u_i,2):\ i\in [m]\}$ is an independent set of $G$ and $X' = \{(u_i,3):\ i\in [m]\}$ is a $2$-packing of $G$. Since $|X| = |X'| = m$, we get  $\pch(G) \le mn - 2m + 2$, thus $\pch(G) = mn - 2m + 2$. 
\qed
\end{proof}

\begin{theorem}
\label{thm:pchu-hamming}
If $n\ge m\ge 3$, then 
$$\pchu(K_m, K_n) = mn - 2m   + 2\,,$$ 
and if $m\ge n\ge 3$, then
$$\pchu(K_m, K_n) = mn - m - n + 2\,.$$ 
\end{theorem}

\begin{proof}
If $n\ge m\ge 3$, then the assertion follows by Lemma~\ref{lem:pch-of-hamming}. 

Assume in the rest of the proof that $m\ge n\ge 3$. Let $f\in K_n^{K_m}$ and set $G = K_m \otimes_f K_n$. We are going to show that $\pch(G) \le mn - m - n + 2$ and consider two cases. 

Assume first that $f$ is not surjective, where,  without loss of generality, $n$ is not in the range of $f$. Then we proceed as in Case 1 of the proof of Lemma~\ref{lem:pch-of-hamming} to infer that $G$ contains a disjoint $2$-packing and an independent set both of size $m$. Hence $\pch(G) \le mn - 2m + 2 \le mn - m - n + 2$ as required. 

Assume second that $f$ is surjective. Then we can follow the argument of Case 1 of the proof of Lemma~\ref{lem:pch-of-hamming} to see that $G$ contains a disjoint $2$-packing of size $n$ and an  independent set of size $m$, implying that $\pch(G) \le mn - m - n + 2$.

To complete the proof we need to show that there exists a function $f$ such that $\pch(K_m \otimes_f K_n) = mn - m - n + 2$. For this sake let $f':V(K_m)\rightarrow V(K_n)$ be defined by setting $f'(u_i)=i$ for each $i\in [n]$, and $f'(u_i)=1$ for each $i>n$. Let $G'=K_m \otimes_{f'} K_n$. Consider an arbitrary $2$-packing $M'$ of $G$. 

\medskip\noindent
{\bf Fact 1}: If $M' \cap \{(u_1,1), (u_{n+1},1), \ldots, (u_{m},1)\} \neq \emptyset$, then 
$$|M' \cap (V(u_1K_n) \cup V(u_{n+1}K_m) \cup \cdots \cup V(u_{m}K_m))| = 1\,.$$
Assume that $(v_i, 1)\in M'$, where $i\in \{1, n+1, \ldots, m\}$. Then every vertex from $X = V(u_1K_n) \cup V(u_{n+1}K_m) \cup \cdots \cup V(u_{m}K_m))$ is at distance at most $2$ from $(v_i, 1)$. Since $M'$ is a $2$-packing, $(v_i, 1)$ is the only vertex from $M'\cap X$ which proves Fact 1. 

\medskip\noindent
{\bf Fact 2}: If $j\in \{2,\ldots, n\}$, then 
$$|M' \cap \{(u_1,j), (u_{n+1},j), \ldots, (u_{m},j)\}| \le 1\,.$$
Assume that $(u_i, j)\in M'$, where $i\in \{1, n+1, \ldots, m\}$. If $(u_{i'}, j)$ is an arbitrary  vertex with $i'\in \{1, n+1, \ldots, m\}\setminus \{i\}$, then $(u_j,1)$ is a common neighbor of $(u_i, j)$ and $(u_{i'}, j)$. Hence $(u_{i'}, j)\notin M'$, and so $(u_i, j)$ is the only vertex from $M' \cap \{(u_1,j), (u_{n+1},j), \ldots, (u_{m},j)\}$. This proves Fact 2. 

\medskip
We are now ready to prove that $|M'|\le n$. We distinguish the following two cases.

\medskip\noindent
{\bf Case 1}: $M' \cap \{(u_1,1), (u_{n+1},1), \ldots, (u_{m},1)\} = \emptyset$.\\
In this case, the vertices of $M'$ lie in fibers $u_2K_n, \ldots, u_{n}K_n$. As each of these layers contains at most one vertex from $M'$ we have $|M'| \le n-1$. 

\medskip\noindent
{\bf Case 2}: $M' \cap \{(u_1,1), (u_{n+1},1), \ldots, (u_{m},1)\} \ne \emptyset$.\\
In this case, $|M' \cap (V(u_1K_n) \cup V(u_{n+1}K_m) \cup \cdots \cup V(u_{m}K_m))| = 1$ by Fact 1. Set $I = \{1,n+1,\ldots, m\}$, and let $(u_i,1)\in M'$, where $i\in I$. Set 
$$t = \Big|M' \cap \bigcup_{i'\in I\setminus \{i\}} V(u_{i'}K_n)\Big|\,.$$
In view of Fact 2, $t\le n-1$. Let now $(u_{i'},j)\in M'$, where $i'\ne i$. Then $j\in \{2,\ldots, n\}$ by Fact 1. Now, $(u_{i'},j)(u_{j},1)\in E(G)$, which implies that $M'\cap V(u_{j}K_n) = \emptyset$. From this and using Fact 2 again we get that among the fibers $u_2K_n, \ldots, u_{n}K_n$, at least $t$ of them have no vertex in $M'$. This in turn implies that $|M'| \le 1 + t + (n-1-t) = n$. We can conclude that $\alpha_2(G') \le n$ and hence $\alpha_2(G') = n$. 

Since $\alpha_2(G') \le n$ and $\alpha(G') \le m$, we have $\pch(G') \ge mn - m - n + 2$. As we already know that $\pch(G') \le mn - m - n + 2$ also holds, we are done. 
\qed
\end{proof}

Until now we have considered the Sierpi\'nski products $K_m \otimes_f K_n$, where $m\ge 3$ and $n\ge 3$. To cover all the cases, we now treat also the cases when $\min \{m,n\} = 2$.

The case when $m = 2$ is straightforward  since in this case, for any function $f$, the graph $K_m \otimes_f K_n$ is a graph obtained from the disjoint union of two copies of $K_n$ by adding an edge between a vertex of one copy and a vertex of the other. This graph $G$ is a diameter $2$ graph with $\alpha(G) = 2$, hence $\pch(G) = 2n-1$ and hence $$\pch(K_2, K_n) = \pchu(K_2, K_n) = 2n - 1\,.$$

Consider now the case $n = 2$ and let $f \in {K_2}^{K_m}$. Let $m_1 = |\{i\in [m]:\ f(u_i) = 1\}|$ and $m_2 = |\{i\in [m]:\ f(u_i) = 2\}|$. Then $m_1, m_2 \in \{0,1,\dots, m\}$, and $m_1 + m_2 = m$. We may without loss of generality assume that $m_1\ge m_2$. Then the graph $K_m \otimes_f K_2$ can be described as follows. It contains cliques $M_1$ and $M_2$ of respective orders $m_1$ and $m_2$, and independent sets $I_1$ and $I_2$ also of respective orders $m_1$ and $m_2$. There is a matching between $M_1$ and $I_1$ and a matching between $M_2$ and $I_2$. Finally, $I_1 \cup I_2$ induce a complete bipartite graph $K_{m_1,m_2}$. Let us denote the described graph by $G_{m_1,m_2}$. See Fig.~\ref{fig:m1=5,m2=3} where the graph $K_8 \otimes_f K_2$ with $m_1 = 5$ and $m_2 = 3$ is drawn, that is, the graph $G_{5,3}$.

\begin{figure}[ht!]
\begin{center}
\begin{tikzpicture}[scale=0.8,style=thick,x=1cm,y=1cm]
\def\vr{4pt}

\begin{scope}[xshift=0cm, yshift=0cm] % K4 1
\coordinate(x1) at (0,0);
\coordinate(x2) at (0,1);
\coordinate(x3) at (0,2);
\coordinate(x4) at (0,3);
\coordinate(x5) at (0,4);
\coordinate(y1) at (2,0);
\coordinate(y2) at (2,1);
\coordinate(y3) at (2,2);
\coordinate(y4) at (2,3);
\coordinate(y5) at (2,4);
\coordinate(z1) at (4,1);
\coordinate(z2) at (4,2);
\coordinate(z3) at (4,3);
\coordinate(w1) at (6,1);
\coordinate(w2) at (6,2);
\coordinate(w3) at (6,3);
% edges		
\draw (x1) -- (x2) -- (x3) -- (x4) -- (x5);
\draw (x1) .. controls (-0.3,0.5) and (-0.3,1.5) .. (x3);
\draw (x2) .. controls (-0.3,1.5) and (-0.3,2.5) .. (x4);
\draw (x3) .. controls (-0.3,2.5) and (-0.3,3.5) .. (x5);
\draw (x1) .. controls (-0.6,1) and (-0.6,2) .. (x4);
\draw (x2) .. controls (-0.6,2) and (-0.6,3) .. (x5);
\draw (x1) .. controls (-0.9,0.5) and (-0.9,3.5) .. (x5);
\draw (w1) -- (w2) -- (w3);
\draw (w1) .. controls (6.3,1.5) and (6.3,2.5) .. (w3);
\foreach \i in {1,...,5}
{ 
\draw (x\i) -- (y\i);
\draw (y\i) -- (z1);
\draw (y\i) -- (z2);
\draw (y\i) -- (z3);
}
\foreach \i in {1,...,3}
{ 
\draw (z\i) -- (w\i);
}

%  vertices
\foreach \i in {1,...,5}
{ 
\draw(x\i)[fill=white] circle(\vr);
\draw(y\i)[fill=white] circle(\vr);
}
\foreach \i in {1,...,3}
{ 
\draw(z\i)[fill=white] circle(\vr);
\draw(w\i)[fill=white] circle(\vr);
}
\draw[rounded corners] (-1, -0.5) rectangle (0.5, 5);
\draw[rounded corners] (1.5, -0.5) rectangle (2.5, 5);
\draw[rounded corners] (3.5, 0.5) rectangle (4.5, 4);
\draw[rounded corners] (5.5, 0.5) rectangle (6.5, 4);
% text
\node at (-0.5,4.5) {$M_1$};
\node at (2,4.5) {$I_1$};
\node at (4,3.5) {$I_2$};
\node at (6,3.5) {$M_2$};

\end{scope}
\end{tikzpicture}
\caption{The Sierpi\'{n}ski product $G_{m_1,m_2} =  K_8 \otimes _f K_2$, where $f(u_i) = 1$ for $i\in [5]$, and $f(u_i) = 2$ otherwise.}
\label{fig:m1=5,m2=3}
\end{center}
\end{figure}
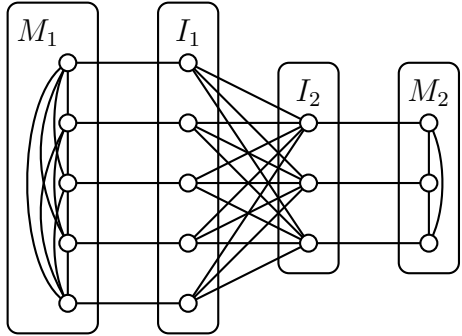

Consider now the representation $G_{m_1,m_2}$ of $K_m \otimes_f K_2$ with $m_1$ and $m_2$ as described above. We may without loss of generality assume that $m_1 \ge m_2$. We infer that $\diam(G_{m_1,m_2}) = 3$, and that $\alpha(G_{m_1,m_2}) = m_1 + 1$ and $\alpha_2(G_{m_1,m_2}) = 2$ as soon as $m_2\ge 1$. These facts imply that if $m_1 \ge m_2$, then  
\begin{equation}
\label{eq:small-cases}
\pch(G_{m_1,m_2}) = 
  \begin{cases}
    2m - m_1 - 1; & m_2\ge 2, \\
    m+1; & m_2 \in \{0,1\}\,.
  \end{cases}
\end{equation}
Determining the extreme cases in~\eqref{eq:small-cases} we arrive at the following result. 

\begin{proposition}
\label{prop:n=2}
If $m\ge 3$, then 
\begin{align*}
    \pch(K_m, K_2) & = m + 1, \\
    \pchu(K_m, K_2) & = 2m - \left\lceil \frac{m}{2} \right\rceil - 1\,. 
\end{align*}
\end{proposition}

%%%%%%%%%%%%%%%%%%%%%%%%%%%%%%%%%%%%%%%%%%%%%%%%%%
\section{Sierpi\'{n}ski products of paths and stars}
\label{sec:paths}
%%%%%%%%%%%%%%%%%%%%%%%%%%%%%%%%%%%%%%%%%%%%%%%%%%

As it turned out, to determine the packing chromatic number of a graph is intrinsically difficult. In fact, as proved by Fiala and Golovach~\cite{fiala-2010}, the corresponding decision problem is NP-complete already on the class of trees, a very rare phenomenon. It is therefore certainly worth exploring the (upper) Sierpi\'nski packing chromatic number for Sierpi\'nski products of trees. In this section we consider in respective subsections Sierpi\'nski products of two paths, of a start by a path, of a path by a star, and of two stars. 

%%%%%%%%%%%%%%%%%%%%%%%%%%%%%%%%%%%%%%%%%%%%%%%%%%
\subsection{Paths by paths}
%%%%%%%%%%%%%%%%%%%%%%%%%%%%%%%%%%%%%%%%%%%%%%%%%%

The determination of the packing chromatic number of the infinite square grid has a long and exciting history. The final value 15 has been recently determined by Subercaseaux and Heule~\cite{suber-2023} using a very subtle use of computer. The value coincides with the upper bound previously obtained in~\cite{martin-2017}. This motivated us to investigate what can be said about $\pch(P_m, P_n)$ and $\pchu(P_m, P_n)$. It turns out that the Sierpi\'{n}ski packing chromatic number is not difficult, while the upper Sierpi\'{n}ski packing chromatic number is more subtle. 

\begin{theorem}
\label{thm:two paths n>=m}
If $n\ge m\ge 3$, then 
$\pch(P_m, P_n) = 3$. 
\end{theorem}

\proof
For any $f\in P_n^{P_m}$, the Sierpi\'{n}ski product $P_m \otimes_f P_n$ contains $P_4$ as a subgraph, hence $\pch(P_m, P_n) \ge 3$. On the other hand, let $V(K_m)=\{u_1,\dots, u_m\}$, $V(K_n) = [n]$, and let $g\in P_n^{P_m}$ be defined as
\[
g(u_i) = \left.
  \begin{cases}
    1; & i \bmod 4 \in \{1,2\}, \\
    n; & i \bmod 4 \in \{3,0\}.
  \end{cases}
  \right.
\]
Then we infer that $P_m \otimes_g P_n \cong P_{mn}$, hence $\pch(P_m \otimes_g P_n) = 3$. We  conclude that $\pch(P_m, P_n) = 3$. 
\qed

Let ${\cal T}$ be the class of all trees $T$ with the following structure. $T$ contains a path on consecutive vertices $v_1, \ldots, v_k$, and to each $v_i$ up to two disjoint paths are attached. We say that the path on the vertices $v_1, \ldots, v_k$ is the {\em spine} of $T$. 

\begin{lemma}
\label{lem:trees-T}
If $T\in {\cal T}$, then $\pch(T)\le 7$. 
\end{lemma}

\proof
Let $v_1, \ldots, v_k$ be the consecutive vertices of the spine of $T$. Color the vertices of the spine using the pattern $1, 4, 1, 5, 1, 6, 1, 7, 1, 4, 1, 5, 1, 6, 1, 7, \ldots$, and let $c(v_i)$ be this color at $v_i$. For $i\in [k]$, let $Q_i$ and $Q_i'$ be the paths attached to $v_i$, where it is possible that one or both of $Q_i$ and $Q_i'$ are isomorphic to $P_1$. Color the vertices of $Q_i$ by the pattern $c(v_i), 
2, 1, 3, 1, 2, 1, 3, 1, \ldots$, and the vertices of $Q_i'$ by the pattern $c(v_i), 
3, 1, 2, 1, 3, 1, 2, 1, \ldots$ It is now straightforward to check that $c$ is a packing coloring of $T$, hence the conclusion. 
\qed

If $T$ is a tree and $x,y\in V(T)$, then let $x \rightsquigarrow y$ denote the unique ($x,y$)-path in $T$. In the case when $xy\in E(T)$, we simplify this notation to $x\rightarrow y$.

\begin{lemma}
\label{lem:path-by-path-is-in-T}
If $n\ge 2$, $m\ge 2$, and $f\in P_n^{P_m}$, then $P_m \otimes_f P_n\in {\cal T}$. 
\end{lemma}

\proof
Let $V(P_m)=\{u_1,\dots, u_m\}$, $V(P_n) = [n]$, $f\in P_n^{P_m}$, and let $G=P_m \otimes_f P_n$. Let $Q$ be the following path in $G$: 
\begin{align*}
(u_1,f(u_2)) & \rightarrow (u_2,f(u_1)) \rightsquigarrow (u_2,f(u_3)) \\
& \rightarrow (u_3,f(u_2)) \rightsquigarrow (u_3,f(u_4))  \\
& \ \ \vdots \\
& \rightarrow (u_{m-1},f(u_{m-2})) \rightsquigarrow (u_{m-1},f(u_{m})) \\
& \rightarrow (u_{m},f(u_{m-1}))\,. 
\end{align*}
Then we claim that $G\in {\cal T}$, where $Q$ is its spine. For this sake note first that if $j\in \{2,\ldots, m-1\}$, then each internal vertex of the subpath $(u_j,f(u_{j-1})) \rightsquigarrow (u_j,f(u_{j+1}))$ is of degree $2$. (It is possible that $f(u_{j-1}) = f(u_{j+1})$, in which case there is no such internal vertex.) Moreover, each of the vertices of $Q$ written in its definition, is either of degree $2$, or has one or two pendant paths attached to it. We conclude that $G\in {\cal T}$.
\qed

Combining the above results brings us to: 

\begin{theorem} 
\label{thm:paths-paths-upper}
If $n\ge 3$ and $m\ge 12$, then $\pchu(P_m, P_n) \in \{6,7\}$. 
\end{theorem} 

\proof
From Lemmas~\ref{lem:trees-T} and~\ref{lem:path-by-path-is-in-T} we deduce that $\pchu(P_m, P_n) \le 7$ for every $n\ge 2$ and $m\ge 2$. 

Let now $n = 3$, $m\ge 12$, $V(P_m)=\{u_1,\dots, u_m\}$, $V(P_3) = [3]$, and let $g\in P_3^{P_m}$ be defined by $g(u_i) = 2$ for $i\in [m]$. Then it is clear that $P_m \otimes_g P_3 \cong P_m\odot 2K_1$, hence Theorem~\ref{thm:laiche} implies that $\pch(P_m \otimes_g P_3) = 6$. This in turn implies that $\pchu(P_m, P_3) \ge 6$. If $n> 3$, then, using the same function $g$ as above, we see that $P_m \otimes_g P_3$ is an induced subgraph of $P_m \otimes_g P_n$, hence $\pchu(P_m, P_n) \ge 6$ holds for any $n\ge 3$. 

Combining the above findings we conclude that $\pchu(P_m, P_n) \in \{6,7\}$. 
\qed

Note that if $T\in {\cal T}$, then $\Delta(T)\le 4$. From Sloper's result~\cite[Theorem 15]{sloper-2004}, it follows that if $T$ is a tree with $\Delta(T) \leq 3$, then $\chi_\rho(T) \leq 7$, see~\cite[Theorem 2.18]{bresar-2020}. For a general tree $T$, however, we know by~\cite[Theorem 5.4]{goddard-2008} that $\chi_\rho(T) \leq (n(T)+7)/4$, except when $n(T)\in\{4, 8\}$, and that the bounds are sharp.

\iffalse
\begin{theorem}
\label{thm:two paths upper}
If $m\ge 35$ and $n\ge 5$, then $\pchu(P_m, P_n) = 7$. 
\end{theorem}

\proof
By Lemma~\ref{lem:path-by-path-is-in-T} we have $P_m \otimes_f P_n\in {\cal T}$ and thus $\pchu(P_m, P_n) \le 7$ by Lemma~\ref{lem:trees-T}. To prove $\pchu(P_m, P_n) \ge 7$, consider the constant function $f_3(u) = 3$ for each $u\in V(P_m)$. Then apply Sloper who in his paper says: 

"A simple computer analysis shows that a color-sequence of length 34 (2342562342
5326423524 6235243265 2342) using only the colors 2 through 6 exists and can be
used for caterpillars with body-length 35 or less, though no such color-sequence using
only colors 2 through 6 has length 35 or greater.''

Then, for each possible $6$-packing coloring $c$ of $P_m \otimes_{f_3} P_n$ there exists a vertex $v_i$ from the spine such that $c(v_i) = 1$. From here on I hope it can be proved that we cannot finish coloring with six colors. TO DO
\qed
\fi

%%%%%%%%%%%%%%%%%%%%%%%%%%%%%%%%%%%%%%%%%%%%%%%%%%
\subsection{Stars by paths}
%%%%%%%%%%%%%%%%%%%%%%%%%%%%%%%%%%%%%%%%%%%%%%%%%%

% Next we focus on products of paths and stars. 
% In the rest of this paragraph, by $S_{i_1,i_2,\dots, i_k}$ we mean the graph obtained from the star $K_{1,k}$ by replacing leaves of the star with paths on $i_1, i_2, \dots, i_k$ vertices, respectively. 
Here we investigate the case when the base graph is a star $K_{1,m}$ and the fiber is a path $P_n$. Since $K_{1,2}\simeq P_3$, to avoid trivial cases, we may assume that $m\ge 3$ and $n\ge 2$.

\iffalse
\begin{theorem} \label{pchu in 6,7-old}
Let $m\ge 3$, $n\ge 2$, $f\in P_n^{K_{1,m}}$, and let $G_f=K_{1,m} \otimes_f P_n$. Then $\pch(G_f)=3$, and $\pchu(G_f)\le 7$. In addition, if $m>2n$ and $n\ge 12$, then $\pchu(G_f)\in\{6,7\}$.    
\end{theorem}
\fi

\begin{theorem} \label{pchu in 6,7}
If $m\ge 3$ and $n\ge 2$, then $\pch(K_{1,m}, P_n)=3$ and $\pchu(K_{1,m}, P_n)\le 7$. In addition, if $m>2n \ge 24$, then $\pchu(K_{1,m}, P_n)\in\{6,7\}$. 
\end{theorem}

\proof
Let $V(K_{1,m})=\{u_1,\dots, u_{m+1}\}$ with $u_1$ as the center, and let $V(P_n)=[n]$. 

First we investigate $\pch(K_{1,m}, P_n)$. Since for every $f\in  P_n^{K_{1,m}}$, the graph $K_{1,m} \otimes_f P_n$ contains $P_4$ as a subgraph, we have $\pch(K_{1,m}, P_n) \ge 3$. On the other hand, consider a constant function  $g\in  P_n^{K_{1,m}}$ with $g(u_i)=1$. Let $x$ be the vertex of $G_f$ of degree $m+1$, and note that $G_f$ consists of $m+1$ paths attached to $x$.  We color $x$ with $2$, vertices at odd distance from $x$ with $1$, vertices at distance $2 \mbox{(mod 4)}$ from $x$ with $3$ and vertices at distance $0 \mbox{(mod 4)}$ from $x$ with color $2$.  Obviously, such a coloring is a packing coloring of $G_g$. Hence we also have $\pch(K_{1,m}, P_n) \le 3$ and can conclude that $\pch(K_{1,m}, P_n) = 3$.

\smallskip

Now we investigate $\pchu(K_{1,m}, P_n)$. Consider an arbitrary function $f\in  P_n^{K_{1,m}}$. Then $K_{1,m} \otimes_f P_n$ is the graph consisting of the path $u_1P_n$ to which $m$ paths of length $n$ are attached at some vertices. (Possibly no path or more than one path is attached to each vertex of $u_1P_n$.) Then we color the vertices of $u_1P_n$ with the pattern 
\begin{quote}
$2,4,3,5,2,6,3,7,2,4,3,5,2,6,3,7,\dots$, 
\end{quote}
and the attached paths with patterns
\begin{quote}
$1,2,1,3,1,2,1,3,\dots$ \quad and \quad $1,3,1,2,1,3,1,2,\dots$,
\end{quote}
where the first pattern is used on those paths whose attached vertex is colored with a color from $\{3,4,5,6,7\}$, and the second pattern when the vertex of attachment is colored with $2$. We infer that such a coloring is a packing coloring of $K_{1,m} \otimes_f P_n$. Indeed, since the latter pattern is used only for vertices of $u_1P_n$ with color $2$, colors $2$ and $3$ are never adjacent on $u_1P_n$, and the vertices with color $3$ on distinct paths with the same vertex of attachment are always at distance at least $2$ from $u_1P_n$, that is, at least $4$ apart. We may conclude that $\pchu(K_{1,m}, P_n) \le 7$.

Assume now that $m > 2n \ge 24$. Consider a function $f\in  P_n^{K_{1,m}}$ such that $f(u_1)=1$, $f(u_{2i})=f(u_{2i+1})=i$, $i \in [n]$, and the remaining vertices of $K_{1,m}$ have arbitrary function values. Then $K_{1,m} \otimes_f P_n$ contains $P_{n}\odot 2K_1$ as an induced subgraph. Since $n\geq 12$, Theorem~\ref{thm:laiche} implies that $\pch(K_{1,m} \otimes_f P_n)\geq 6$ and consequently $\pchu(K_{1,m}, P_n) \geq 6$. 
\qed

%%%%%%%%%%%%%%%%%%%%%%%%%%%%%%%%%%%%%%%%%%%%%%%%%%
\subsection{Paths by stars}
%%%%%%%%%%%%%%%%%%%%%%%%%%%%%%%%%%%%%%%%%%%%%%%%%%

We now deal with the case when the base graph is a path $P_m$ and the fiber is a star $K_{1,n}$. Again, since $K_{1,2}\simeq P_3$, to avoid trivial cases, we consider $n\ge 3$ and $m\ge 2$.

\iffalse
\begin{theorem}
Let $m\ge 2$, $n\ge 3$, $f\in K_{1,n}^{P_m}$, and let $G_f= P_m \otimes_f K_{1,n}$. Then $\pch(G_f)=3$. 
\end{theorem}
\fi

\begin{theorem}
\label{thm:path-star}
If $m\ge 2$ and $n\ge 3$, then $\pch(P_m, K_{1,n}) = 3$ and $\pchu(P_m, K_{1,n}) \le 9$. Moreover, if  $n=3$ and $m\ge 10$, then $\pchu(P_m, K_{1,n}) \in \{6,7,8,9\}$, and if $n>3$ and $m\ge 35$, then $\pchu(P_m, K_{1,n}) \in \{7,8,9\}$.
\end{theorem}

\proof
Let $V(P_m)=\{u_1,\dots, u_{m}\}$, and let $V(K_{1,n})=[n+1]$, where $1$ denotes the center of the star. Since for any $f\in  K_{1,n}^{P_m}$, the Sierpi\'nski product $P_m \otimes_f K_{1,n}$ contains a $P_4$ as a subgraph, we have $\pch(P_m, K_{1,n}) \ge 3$. Now consider a function  $g\in K_{1,n} ^{P_m}$ defined by
\[
g(u_i) = \left.
  \begin{cases}
    1; & i\equiv 1 (\mbox{mod }4), \\
    2; & i\equiv k (\mbox{mod } 8),\ k\in\{2,4\}, \\  
    3; & i\equiv k (\mbox{mod }8),\ k\in\{6,0\}, \\ 
    4; & i\equiv 3 (\mbox{mod }4).
  \end{cases}
  \right.
\]
The structure of $P_m \otimes_g K_{1,n}$ is illustrated in Fig~\ref{Fig path-star} for the case $m=14$ and $n=3$.

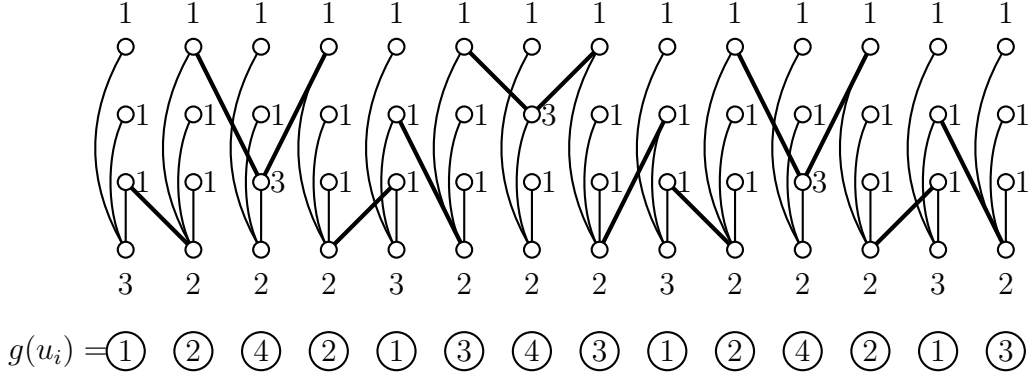
\begin{figure}[ht!]
\begin{center}
\begin{tikzpicture}[scale=0.60,style=thick,x=1.5cm,y=1.5cm]
\def\vr{5pt}

\begin{scope}[xshift=0cm, yshift=0cm] % K4 1
\coordinate(a1) at (0,0);
\coordinate(b1) at (0,1);
\coordinate(c1) at (0,2);
\coordinate(d1) at (0,3);
\coordinate(a2) at (1,0);
\coordinate(b2) at (1,1);
\coordinate(c2) at (1,2);
\coordinate(d2) at (1,3);
\coordinate(a3) at (2,0);
\coordinate(b3) at (2,1);
\coordinate(c3) at (2,2);
\coordinate(d3) at (2,3);
\coordinate(a4) at (3,0);
\coordinate(b4) at (3,1);
\coordinate(c4) at (3,2);
\coordinate(d4) at (3,3);
\coordinate(a5) at (4,0);
\coordinate(b5) at (4,1);
\coordinate(c5) at (4,2);
\coordinate(d5) at (4,3);
\coordinate(a6) at (5,0);
\coordinate(b6) at (5,1);
\coordinate(c6) at (5,2);
\coordinate(d6) at (5,3);
\coordinate(a7) at (6,0);
\coordinate(b7) at (6,1);
\coordinate(c7) at (6,2);
\coordinate(d7) at (6,3);
\coordinate(a8) at (7,0);
\coordinate(b8) at (7,1);
\coordinate(c8) at (7,2);
\coordinate(d8) at (7,3);
\coordinate(a9) at (8,0);
\coordinate(b9) at (8,1);
\coordinate(c9) at (8,2);
\coordinate(d9) at (8,3);
\coordinate(a10) at (9,0);
\coordinate(b10) at (9,1);
\coordinate(c10) at (9,2);
\coordinate(d10) at (9,3);
\coordinate(a11) at (10,0);
\coordinate(b11) at (10,1);
\coordinate(c11) at (10,2);
\coordinate(d11) at (10,3);
\coordinate(a12) at (11,0);
\coordinate(b12) at (11,1);
\coordinate(c12) at (11,2);
\coordinate(d12) at (11,3);
\coordinate(a13) at (12,0);
\coordinate(b13) at (12,1);
\coordinate(c13) at (12,2);
\coordinate(d13) at (12,3);
\coordinate(a14) at (13,0);
\coordinate(b14) at (13,1);
\coordinate(c14) at (13,2);
\coordinate(d14) at (13,3);
% edges		

\foreach \i in {1,...,14}
{ 
\draw (a\i) -- (b\i);
\draw (a\i) .. controls (-0.3+\i-1,0.5) and (-0.3+\i-1,1.5) .. (c\i);
\draw (a\i) .. controls (-0.6+\i-1,1) and (-0.6+\i-1,2) .. (d\i);
}
\draw[ultra thick] (b1) -- (a2); 
\draw[ultra thick] (d2) -- (b3); 
\draw[ultra thick] (b3) -- (d4); 
\draw[ultra thick] (a4) -- (b5); 
\draw[ultra thick] (c5) -- (a6); 
\draw[ultra thick] (d6) -- (c7); 
\draw[ultra thick] (c7) -- (d8); 
\draw[ultra thick] (a8) -- (c9); 
\draw[ultra thick] (b9) -- (a10); 
\draw[ultra thick] (d10) -- (b11); 
\draw[ultra thick] (b11) -- (d12); 
\draw[ultra thick] (a12) -- (b13); 
\draw[ultra thick] (c13) -- (a14); 

\foreach \i in {0,...,13}
{
\draw (\i,-1.5)[fill=white] circle(12pt); 
}

%  vertices
\foreach \i in {1,...,14}
{ 
\draw(a\i)[fill=white] circle(\vr);
\draw(b\i)[fill=white] circle(\vr);
\draw(c\i)[fill=white] circle(\vr);
\draw(d\i)[fill=white] circle(\vr);
}
% text
\node at (-1,-1.5) {$g(u_i)=$}; 
\node at (0,-1.5) {$1$};
\node at (1,-1.5) {$2$};
\node at (2,-1.5) {$4$};
\node at (3,-1.5) {$2$};
\node at (4,-1.5) {$1$};
\node at (5,-1.5) {$3$};
\node at (6,-1.5) {$4$};
\node at (7,-1.5) {$3$};
\node at (8,-1.5) {$1$};
\node at (9,-1.5) {$2$};
\node at (10,-1.5) {$4$};
\node at (11,-1.5) {$2$};
\node at (12,-1.5) {$1$};
\node at (13,-1.5) {$3$};

\foreach \i in {0,...,13}
{
\node at (\i,3.5) {$1$};
}
\foreach \i in {0,...,5,7,8,...,13}
{
\node at (\i+.25,2) {$1$};
}
\foreach \i in {0,1,3,4,...,9,11,12,13}
{
\node at (\i+.25,1) {$1$};
}
\foreach \i in {1,2,3,5,6,7,9,10,11,13}
{
\node at (\i,-0.5) {$2$};
}
\node at (6.25,2) {$3$};
\node at (2.25,1) {$3$};
\node at (10.25,1) {$3$};
\node at (0,-.5) {$3$};
\node at (4,-.5) {$3$};
\node at (8,-.5) {$3$};
\node at (12,-.5) {$3$};

\end{scope}
\end{tikzpicture}
\caption{The Sierpi\'{n}ski product $P_{14} \otimes_g K_{1,3}$.}
\label{Fig path-star}
\end{center}
\end{figure}

Color all leaves and all vertices of degree $2$ of $P_m \otimes_g K_{1,n}$ with color 1, all vertices $(u_i,1)$ where $g(u_i)=1$ with color $3$, and all the remaining vertices with color $2$, see Fig.~\ref{Fig path-star} again. Note that vertices with color $2$ are mutually at distance at least $3$, and vertices colored with color $3$ are pairwise at distance at least $4$. Hence such a coloring is a packing coloring and we can conclude that  $\pch(P_m, K_{1,n}) = 3$. 

We next consider $\pchu(P_m, K_{1,n})$. For this sake, let $g\in  K_{1,n}^{P_m}$ be a function defined by $g(u_i)=1$, $i\in [m]$. Then $P_m \otimes_g K_{1,n}$ is isomorphic to $P_m \odot nK_1$. If $n=3$, then Theorem~\ref{thm:laiche 3K_1} implies that $\pchu(P_m, K_{1,n})\ge 6$ whenever $m\ge 10$. And if $n>3$, then Theorem \ref{thm:laiche pK_1} gives $\pchu(P_m, K_{1,n})\ge 7$. 

It remains to prove that $\pchu(P_m, K_{1,n}) \le 9$. Let $f\in K_{1,n}^{P_m}$ be arbitrary and let $Q$ be a shortest path in $P_m \otimes_f K_{1,n}$ between a an arbitrary, fixed vertex of $u_1K_{1,n}$, and a an arbitrary, fixed vertex of $u_mK_{1,n}$. %Then $G_{g'}$ consists of $Q$, $2$ copies of $K_{1,n}$, each of them attached to exactly one vertex of $u_1K_{1,n}$ and of $u_mK_{1,n}$, and $m-2$ copies of the star $K_{1,n-1}$, each of them attached to exactly two consecutive internal vertices of $Q$. Some of the vertices of attachment of the stars $K_{1,3}$ are centers, some of the vertices of attachment of the stars are leaves. 
We define a coloring $c$ of $P_m \otimes_f K_{1,n}$ in the following way. First color the vertices of $Q$ with the pattern $$ \begin{array}{l} 3,4,7,5,3,6,4,9,3,5,7,4,3,6,8,5,3,4,7,9,3,5,4,6,3,8,7,4,3,5,6,9,\\ 3,4,7,5,3,6,4,8,3,5,7,4,3,6,9,5,3,4,7,8,3,5,4,6,3,9,7,4,3,5,6,8\,,\end{array}$$ 
which was found and verified by a computer. Then color all leaves with $1$, and finally color all the remaining vertices with $2$. Note that the vertices with color $2$ are centers of those stars $u_iK_{1,n}$ which do not belong to $Q$ and hence they are pairwise at distance at least $3$ apart. We can conclude that $c$ is a packing coloring and since $f$ was arbitrary, we can conclude that $\pchu(P_m, K_{1,n})\le 9$.
%if $H_i$ is attached to $P_m$ at the center of $K_{1,n}$, then we color leaves of $u_iK_{1,n}$ with color $1$. If $u_iK_{1,n}$ is attached to $P_m$ at a leaf, we color the center of $u_iK_{1,n}$ with color $2$ and the remaining leaves of $u_iK_{1,n}$ with color $1$. One can easily see that such a coloring is a packing coloring of $G_f$, since color $1$ is used only on leaves of $G_f$ and color $2$ only at centers of $u_iK_{1,n}$ which are mutually at distance at least $3$ apart.
%\sk{MAYBE A PICTURE???}
\qed

%%%%%%%%%%%%%%%%%%%%%%%%%%%%%%%%%%%%%%%%%%%%%%%%%%
\subsection{Stars by stars}
%%%%%%%%%%%%%%%%%%%%%%%%%%%%%%%%%%%%%%%%%%%%%%%%%%

As the last step, we investigate the (upper)  Sierpi\'{n}ski packing chromatic number of the product of two stars $K_{1,m}$ and $K_{1,n}$. Again, to avoid trivial cases, we consider $m\ge 3$ and $n\ge 3$.

\begin{theorem}
\label{thm:star-by-star}
If $m\ge 3$ and $n\ge 3$, then $\pch(K_{1,m}, K_{1,n}) = 3$ and 
$\min \{m+2, n+2\} \le \pchu(K_{1,m}, K_{1,n}) \le \max \{m+2,n+2\}$. 
\end{theorem}

\proof
Let $V(K_{1,m})=\{u_1,\dots, u_{m+1}\}$ with $u_1$ as the center of $K_{1,m}$, and let $V(K_{1,n})=[n+1]$ with $1$ as the center of $K_{1,n}$. 

Since for any $f\in  K_{1,n}^{K_{1,m}}$, the Sierpi\'nski product $K_{1,m} \otimes_f K_{1,n}$ contains $P_4$ as a subgraph, we have  $\pch(K_{1,m}, K_{1,n})\geq 3$. Consider a constant function  $g\in K_{1,n}^{K_{1,m}}$ with $g(u_i)=n+1$, $i\in [m+1]$. Color the vertex $(u_1,n+1)$ of $K_{1,m} \otimes_g K_{1,n}$ with color $3$, the vertices $(u_i,1)$, $i\in [m+1]$, with color $2$, and all the remaining vertices of with color $1$. Since this coloring is a packing coloring of $K_{1,m} \otimes_g K_{1,n}$, we can conclude that $\pch(K_{1,m}, K_{1,n}) = 3$.

\medskip
Now we focus on $\pchu(K_{1,m}, K_{1,n})$. Consider the Sierpi\'nski product $K_{1,m} \otimes_g K_{1,n}$, where $g(u_i) = 1$, $i\in [m+1]$. This graph consists of a star $K_{1,m+n}$ with the center $(u_1,1)$ and $m$ copies of the star $K_{1,n}$, each of them attached to a private leaf of $K_{1,m+n}$. Now we color vertices of $K_{1,m} \otimes_g K_{1,n}$. If $(u_1,1)$ is colored with color $1$, then all its neighbors must have mutually distinct colors different from $1$, implying that we need at least $m+n+1$ colors. Also, if $(u_i,1)$, $i > 1$, is colored with color $1$, then (again) all its $n+1$ neighbours must have mutually distinct colors different from $1$, implying that we need at least $n+2$ colors. The last case is when all the vertices $(u_i,1)$, $i \in [m+1]$, receive color different from $1$. Then these vertices must receive different colors and hence such a coloring requires at least $m+2$ colors. Thus $\pchu(K_{1,m}, K_{1,n})\ge \min\{m+2,n+2\}$.

Let $f\in  K_{1,n}^{K_{1,m}}$ be arbitrary, $f_i$ denote number of vertices of $K_{1,m+1}$ with function value $i$, $i\in\{1,2,3,\dots, n+1\}$, and let $H_j$ denote the subgraph of $G_f$ induced by $V(u_jK_{1,n})$ ($j\in\{1,2,\dots, m\}$), i.e.,  the fiber $u_iK_{1,n}$. For proving the upper bound we consider the following two possibilities:

\smallskip\noindent
{\bf Case 1}: $f(u_1)\not= 1$. \\
%Note that the connecting vertex in each $V(u_iK_{1,n})$ ($i\in \{2,3,\dots, m+1\})$ is of degree $2$. In this case, $G_f$ has the following structure. Consider a star $K_{1,n+f_1}$. We denote its center by $x$, and its leaves by $x_1,\dots, x_n, y_1, \dots, y_{f_1}$. Now we add $f_1$ copies of $K_{1,n+1}$ by unifying vertex $y_j$ with a leaf of the $j$-th copy of $K_{1,n+1}$ ($j\in\{1,2,\dots, f_1$), one copy of $K_{1,n+1}$ to one vertex of $G_1$. We further add $f_k$ copies of $K_{1,n+1}$ by joining by an edge the vertex $y_k$ with a leaf of each of the $f_k$ copies of $K_{1,n+1}$, $k\in\{2,3,\dots, n+1$ (possibly, if $f_j=0$ for some $j$, then we join no copy of $K_{1,n+1}$). We denote such a graph by $G_1$. We now color all leaves of $G_1$ with color $1$ and all vertices adjacent to any leaf with $2$ except vertex $x$. We futher color all neighbors of vertices $y_j$ of degree at least $3$ ($j\in\{1,2,\dots, f_1\}$), all vertices $y_{j'}$  of degree at most $2$ ($j'\in\{1,2,\dots, f_1\}$) and all vertices $x_{j''}$ ($j''\in\{1,2,\dots, n\}$) with color $1$. Finally we color neighbors of degree $2$ of all vertices $y_{j'}$  of degree at most $2$ ($j'\in\{1,2,\dots, f_1\}$) with color $3$, and all remaining vertices - vertex $x$ and all vertices $y_j$ of degree at least $3$ ($j\in\{1,2,\dots, f_1\}$) with mutually distinct colors different from $\{1,2,3\}$. Clearly, such a coloring is an $(n+4)$-packing coloring of $G$.
Note that the connecting vertex of each $u_jK_{1,n}$, $j\in \{2,3,\dots, m+1\}$, is of degree $2$. In each $u_jK_{1,n}$, $j>1$, we color the vertex $(u_j,1)$ with color $2$, and all the remaining vertices with color $1$. In $u_1K_{1,n}$, we color all leaves which are not connecting vertices with color $1$, and the remaining vertices with mutually distinct colors $3,4,\dots$ This is a packing coloring which uses most $n+3$ colors. If $f$ is surjective, then clearly $m>n$ and $m+2\ge n+3$. On the other hand, if $f$ is not surjective, e.g.\ $j$ is not in the range of $f$ for some $j \in [n+1]$, then we can color the vertex $(u_1,j)$ with color $1$. In any case, we used at most $n+2$ colors.

\smallskip\noindent
{\bf Case 2}: $f(u_1)= 1$. \\
In this case we infer that the vertices $V(K_{1,m}) \times \{2,\ldots, n+1\}$ form an independent set, hence we can color each of then with color $1$. The remaining $m+1$ vertices can be colored with distinct colors $2,3,\dots, m+1$. Thus we have a packing coloring using $m+2$ colors.
\qed

%%%%%%%%%%%%%%%%%%%%%%%%%%%%%%%%%%%%%%%%%%%%%%%%%%
\section{Recognition of Sierpi\'{n}ski products with acyclic base graphs}
\label{sec:recognition}
%%%%%%%%%%%%%%%%%%%%%%%%%%%%%%%%%%%%%%%%%%%%%%%%%%

In this section, we demonstrate that it can be checked in polynomial time whether a given graph is isomorphic to a Sierpi\'{n}ski product graph whose both base graph and the fiber graph are trees. 

Let $e = uv$ be a pendant edge of a tree $T$, and consider the Sierpi\'{n}ski product $T \otimes_f H$, where $H$ is an arbitrary graph. Then we say that the connecting edge between the fibers $uH$ and $vH$ is a \emph{pendant connecting edge} of $T \otimes_f H$. 

\begin{lemma}
\label{lem:pendant-conn-edge}
Let $T$ be a tree, $H$ a graph, and $e$ an edge of the Sierpi\'{n}ski product $T \otimes_f H$. Then $e$ is a pendant connecting edge if and only if $e$ is a cut edge such that the two components of $(T \otimes_f H)-e$ have respective orders $n(T)$ and $n(T \otimes_f H) - n(T)$.
\end{lemma}

\proof
Let $e = (x,f(x'))(x',f(x))$ be a a pendant connecting edge of $T \otimes_f H$, where $x\in V(T)$ is a vertex with $\deg_T(x) = 1$. Then $(T \otimes_f H)-e$ is the disjoint union of the fiber $xT$ of $T \otimes_f H$ and the Sierpi\'{n}ski product $(T-x) \otimes_{f|T-x} H$. These two components have orders $n(T)$ and $n(T \otimes_f H) - n(T)$, respectively.

Conversely, assume that $e\in T\otimes_f H$ is a cut edge such that the two components of $(T \otimes_f H)-e$ have respective orders $n(T)$ and $n(T \otimes_f H) - n(T)$. 

Suppose first that $e$ is a Type-1 edge, so that $e = (x,h)(x,h')$ for some $x\in V(T)$ and $hh'\in E(H)$. Since $e$ is a cut edge of $T\otimes_f H$, we infer that $e$ is also a cut edge of the fiber $xH$. If the components of $xH-e$ have orders $s_1$ and $s_2$, then $s_1 < n(H)$ and $s_2 < n(H)$. It follows that the components $(T \otimes_f H) - e$ have orders $t_1 n(T) + s_1$ and $t_2 n(T) + s_2$ (for some appropriate $t_1$ and $t_2$), hence then cannot be of orders $n(T)$ and $n(T \otimes_f H) - n(T)$.  

Suppose second that $e$ is a connecting edge which is not a pendant connecting edge. Then each of the two components of $(T \otimes_f H) - e$ has order at least $2n(T)$. 

We conclude that $e$ must be a pendant connecting edge. 
\qed

Let $X$ be an arbitrary connected graph and say that we wish to decide whether $X$ is a non-trivial Sierpi\'{n}ski product graph whose base graph is a tree. First, this is possible only if $n(X) = n_1n_2$ for some $n_1, n_2\ge 2$. We are going to check all such factorizations of $n(X)$ and assume in the rest that $n(X) = n_1n_2$ is fixed. 
Assume now that that $X\cong T \otimes_f H$ for some non-trivial tree $T$ with $n(T) = n_1$, a non-trivial graph $H$ with $n(H) = n_2$, and some $f\in H^T$. For every edge $e$ of $X$ we check whether $e$ is a cut edge such that $X-e$ consists of two components of respective orders $n_2$ and $n_1n_2 - n_2$. If there is no such edge, then by Lemma~\ref{lem:pendant-conn-edge} $X$ is not a Sierpi\'{n}ski product with the base graph being a tree of order $n_1$. Otherwise, let $e$ be an arbitrary edge of $X$ that fulfils the conditions and let $H$ be the component of $X-e$ of order $n_2$. Then, if $X$ is a Sierpi\'{n}ski product graph whose base graph is a tree, then we must have $X\cong T \otimes_f H$. Let now $X'$ be the graph obtained from $X$ by removing the component $H$ and the edge $e$. Then we can proceed as above to detect a possible pendant connecting edge $e'$ of $X'$ such that (at least) one component of $X'-e'$ has order $n_2$, say $H'$. In order that $X\cong T \otimes_f H$, we must check whether $H'\cong H$ holds. If this is not the case, we conclude that $X$ is not a Sierpi\'{n}ski product graph whose base graph is a tree. Otherwise we proceed as above until we either reject $X$ or we find a factorization of $X$ as $T \otimes_f H$. 

Note that in the above arguing, the number of steps and each of the checking can be done in time polynomial with respect to $n(G)$, except the isomorphism testing. As the later is polynomial for trees, we can formulate the following:

\begin{theorem}
\label{thm:factor-tree-by-tree}
If $X$ is a connected graph, then it can be checked in polynomial time whether there exists trees $T_1$ and $T_2$ and $f\in T_2^{T_1}$ such that $X \cong T_1 \otimes_f T_2$. 
\end{theorem}

\iffalse
%%%%%%%%%%%%%%%%%%%%%%%%%%%%%%%%%%%%%%%%%%%%
\section{Concluding remarks}
\label{sec:conclude}
%%%%%%%%%%%%%%%%%%%%%%%%%%%%%%%%%%%%%%%%%%%%

We have seen that checking whether a given graph is isomorphic to a Sierpi\'{n}ski product whose both facts are trees can be done efficiently. The complexity of determining whether a given graph is isomorphic to a Sierpi\'{n}ski product remains an open and challenging problem. 

EXTEND IT !!!!
\fi

\section*{Acknowledgements}

S.K.\ was supported by the Slovenian Research and Innovation Agency (ARIS) under the grants P1-0297, N1-0285, and N1-0355.

%%%%%%%%%%%%%%%%%%%%%%%%%%%%
\section*{Declaration of interests}
%%%%%%%%%%%%%%%%%%%%%%%%%%%%
 
The authors declare that they have no conflict of interest. 

%%%%%%%%%%%%%%%%%%%%%%%%%%%%
\section*{Data availability}
%%%%%%%%%%%%%%%%%%%%%%%%%%%%
 
Our manuscript has no associated data.

%%%%%%%%%%%%%%%%%%%%%%%%%%%%%%%%%%%%%%%%%%%%%%%%%%%%%%%%%%%%%%%%%%%%%
%%%%%%%%%%%%%%%%%%%%%%%%%%%%%%%%%%%%%%%%%%%%%%%%%%%%%%%%%%%%%%%%%%%%%

\end{document}